\documentclass{amsart}

\usepackage{lineno,hyperref}
\usepackage{amssymb}
\usepackage{amsmath}
\usepackage{amsthm}
\usepackage{mathrsfs}
\usepackage{bbm}
\usepackage{cite}
\usepackage[all]{xy}
\usepackage{graphicx}

\newtheorem{thm}{Theorem}[section]
\newtheorem{cor}[thm]{Corollary}
\newtheorem{defn}[thm]{Definition}

\newtheorem{prop}[thm]{Proposition}

\newtheorem{rem}[thm]{Remark}

\title{Slice regular functions of  several octonionic variables.}
\author{Guangbin Ren, Ting  Yang}
\email[G.~Ren]{rengb@ustc.edu.cn}
\email[T.~Yang]{tingy@mail.ustc.edu.cn}
\date{\today}
\thanks{The first author is supported by the NNSF  of China (11771412).}

\begin{document}
\maketitle
\markboth{Slice regular functions of  several octonionic variables}
{Slice regular functions of  several octonionic variables}
\renewcommand{\sectionmark}[1]{}

\begin{abstract} Octonionic analysis is becoming eminent due to   the role of octonions in the theory of $G_2$ manifold.
In this article, a new slice theory is introduced as a generalization of the holomorphic theory of  several complex variables
to the noncommutative or nonassociative realm.
The  Bochner-Martinelli formula  is established for slice functions  of  several octonionic variables as well as several quaternionic variables.
In this setting, we find    the Hartogs phenomena for slice regular functions.
\end{abstract}

\section{Introduction}
Recently,  octonionic analysis has been put in the spotlight as the development of the theory of $G_2$ manifold since $G_2$ is the automorphism  group of the octonion algebra (see \cite{Sergey2017002, Wang2014001}). In this article, we try to establish a new theory related to  the octonionic analysis. That is the slice  octonionic analysis which generalize the lower dimensional theory to higher dimensional.
It reminds us that the slice technique may be helpful in the study of the big problem of  Ising model in higher dimensions.
Notice that up to now  we only know that  $2D$ Ising model is exactly solvable or integrable. For Ising model and the related discrete analysis, we refer to \cite{Hongler2013002, Zhu2017003,Zhu2018001}.

Now we move on to recall the  classical theory on slice  analysis. The theory of slice regular functions of one quaternionc variable, initiated by  Gentili and Struppa \cite{Gentili2006001,Gentili2007001},  provides an effective approach to generalize the beautiful theory of holomorphic functions of one complex variable to the   non-commutative or even non-associative realm. It turns out to have    potential applications in the quantum theory since it demonstrates that the self-adjoint operators of quaternions  admits  {\it real} $S$-spectrum \cite{Ghiloni2013001}.
Colombo, Sabadini, and Struppa later extend the theory   to the Clifford algebras\cite{Colombo2009003,Gentili2009004}, and Gentilli and  Struppa  \cite{Gentili2010002}  to the octonions.
A further extension to  real alternative  algebras was introduced by Ghiloni and Petrotti in \cite{Ghiloni2011001}, and  later in \cite{Ghiloni2017001}.

The root of this theory lies in its effective approach to construct new functions from the stem function which need not be holomorphic (see \cite{Ghiloni2011002}). This approach goes back to
the well-known Futer contruction (see \cite{Sudbery1979001, Sommen2000001,Klaus2008005}).

Now we recall this construction in more detail (see \cite{Ren2016001}).
We consider the holomorphic function $F$, defined in a domain  $D$ of the complex plan $\mathbb{C}$  invariant under
the complex conjugate, with values in the
complexification  of an alternative algebra $A$ over $\mathbb R$.
Then it admits
a unique slice regular extension $f:\Omega_D\rightarrow \mathbb{A}$ with
$$\Omega_D:=\{\alpha+\beta J \ |\  \alpha +i \beta \in D, \alpha ,\beta \in \mathbb{R} , \ J\in \mathbb{S}_A\}\subseteq
\mathcal{Q}_{\mathbb{A}}$$ such that the following diagram commutes for every  $J\in \mathbb{S}_\mathbb{A}:$
\begin{figure}[htb]
  \centering
  \includegraphics{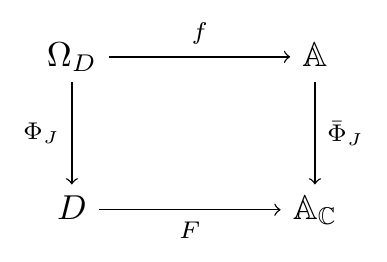}
\end{figure}

\noindent The construction above depends heavily on the so-called slice complex nature of $\mathcal{Q}_{\mathbb{A}},$ the quadratic cone of $\mathbb{A},$ i.e.,
$$\mathcal{Q}_{\mathbb{A}}=\bigcup_{J\in \mathbb{S}_\mathbb{A}}\mathbb{C}_J,$$ and
$$\mathbb{C}_I\cap \mathbb{C}_J=\mathbb{R}$$ for all $I,J\in \mathbb{S}_\mathbb{A}$ with $I\not= \pm J.$
Here $\mathbb{S}_\mathbb{A}$ denotes the set of square roots of $-1$ in the algebra $\mathbb{A},$ i.e.,
$$\mathbb{S}_\mathbb{A}:=\{J\in\mathcal{Q}_\mathbb{A}\ \vert \ J^2=-1\},$$
and the two associated   maps are defined respectively by
$$\Phi_J(a+ib)=a+Jb,\qquad \forall\ a,b\in \mathbb{R},$$
$$\tilde{\Phi}_J(\alpha+i\beta)=\alpha+J\beta,\qquad \forall\ \alpha,\beta \in \mathbb{A}.$$
The preceding approach results in the slice theory which provides an effective generalization  of the theory of {\it one}  complex variable to the setting of quaternions or even more general algebras.

It is quite natural to do such an extension so that to generalize  the theory of several complex variables to the setting of non-commutative or non-associative realm.

%
%

The first attempt was given by  Ghiloni and Perotti  \cite{Ghiloni2012002}.  They  introduced the class of slice regular functions of several Clifford variables. The
definition of the slice functions is based on the concept of stem functions of several variables.

To compare their theory with ours, we need to recall their construction in some details.

Let $\mathbb{R}_n$ denote the real   Clifford algebra of signature $(0,n)$ generated by $e_1,\cdots,e_n.$ Its elements can be expressed as
$$x=\sum_{K\in \mathcal{P}(n)}x_Ke_K,$$
where the coefficients $x_K\in \mathbb{R}$,  the products $e_K:=e_{k_1}e_{k_2}\cdots e_{k_r}$ are the basis elements of the Clifford algebra $\mathbb{R}_n$, and the sum runs over the set
\begin{equation*}
\begin{split}
\mathcal{P}(n)&=\{(k_1,\cdots, k_r)\in \mathbb{N}^r \ | \ r=0, 1,\cdots,n,\ 1\leq k_1< \cdots <k_r\leq n\}.
\end{split}
\end{equation*}
 The unit of the Clifford algebra corresponds to $K=\varnothing$, and we set $e_\phi=1$.

Consider the stem function $$F:D\rightarrow \mathbb{R}_m\otimes \mathbb{R}_n,$$
defined in
 an open set in $\mathbb{C}^n,$ invariant w.r.t. complex conjugation in every variable $z_1,\cdots,z_n.$
If we denote it by
\begin{eqnarray} \label{def:capital-F} F=\sum_{K\in \mathcal{P}(n)}e_KF_K, \qquad F_K:D\rightarrow \mathbb{R}_m, \end{eqnarray} then as a stem function, it satisfies
the Clifford-intrinsic condition. That is, for each $K\in \mathcal{P}(n),\ k\in \{1,\cdots,n\}$,
and $z=(z_1,\cdots,z_n)\in D,$ the components $F_K$ satisfy the compatibility conditions
\[F_K(z_1,\cdots,z_{k-1},\overline{z_k},z_{k+1},\cdots,z_n)=
\begin{cases} F_K(z) &\text{if $k\not \in K$},\\
-F_K(z) & \text{if $k\in K$}.
\end{cases}
\]

Let $\Omega(D) $ be the circular subset of $(\mathcal{Q}_m)^n$ associated to $D\subset \mathbb{C}^n$. More precisely, $\Omega(D)$ consists of all
$x=(x_1,\cdots,x_n)\in (\mathcal{Q}_m)^n$ with
$$x_k=\alpha_k+\beta_kJ_k\in \mathbb{C}_{J_k}$$
for any  $J_k\in \mathbb{S}_m$ and  $\alpha_k,\beta_k\in \mathbb R$ provided $(\alpha_1+i\beta_1,\cdots,\alpha_n+i\beta_n)\in D$.
Here  $\mathcal{Q}_m$ stands for the quadratic cone in $\mathbb R_m$, i.e.,
 $$\mathcal{Q}_m:=\mathbb{R} \cap \{x\in \mathbb{R}_m \ | \ t(x)\in \mathbb{R}, \ n(x)\in \mathbb{R},\ 4n(x)>t(x)^2  \},$$
where $t(x)=x+\bar x$ denotes  the trace of $x$ and $n(x)=x \bar x$ the (squared) norm of a Clifford element $x$.
As usual, we take $\mathbb{S}_m$ in place of $\mathbb S_{\mathbb A}$ in the case of $\mathbb A=\mathbb R_m$.

With the stem function $F$ given by (\ref{def:capital-F}) and  $$x=(x_1,\cdots,x_n)=(\alpha_1+J_1\beta_1,\cdots,\alpha_n+J_n\beta_n)\in\Omega(D),$$
 the slice function $\mathcal{I}(F)$ is defined as
$$\mathcal{I}(F)(x):=\sum_{K\in \mathcal{P}(n)}J_KF_K(\alpha_1+i\beta_1,\cdots,\alpha_n+i\beta_n).$$
Here
\begin{eqnarray}\label{def:complex-structurs} J_K:=\prod_{k\in K}^\rightarrow  J_k= J_{k_1}\cdots J_{k_s},\end{eqnarray}
is the odered product.

When the stem function $F$ is  holomorphic, i.e.,
$$\frac{1}{2}\big(\frac{\partial F}{\partial \alpha_k}+{J}_k \frac{\partial F}{\partial \beta_k}\big)=0, \qquad k=1,\cdots,n. $$
the slice function  $\mathcal{I}(F)$ is called  slice regular on $\Omega_D.$
Many results for slice regular functions  are announced in \cite{Ghiloni2012002}.

Observed that the setting considered in \cite{Ghiloni2012002} is too general,  Colombo, Sabadini, and Struppa
  \cite{Colombo2012003} chose to  move on
 in  some special case about   the slice regular functions with $m=2$  in which they can provide detail proofs.
Moreover,  to get rid of the compatibility   conditions  they restrict  their consideration to the case of the upper half space, i.e., they only consider the domain
$$D\subset(\mathbb{R}\times \mathbb{R}^+)^n.$$

The purpose of this article is to establish the slice theory of several octonionic variables,
which is as a generalization of the theory of several complex variables, instead of the theory of one complex variable.

To overcome the difficulties appearing  in \cite{Ghiloni2012002, Colombo2012003}, We adopt a new  trick by restrict our attention to
the same complex structure $J$, in contrast to the classical case
where the imaginary units may be distinct. Our approach makes many results of several complex variables extended  to the non-commutative or non-associative setting with the help of the theory of stem functions.
In particular, we establish
the Bochner-Martinelli formula for slice functions   and Hartogs theorem for slice regular functions in several octonionic variables as well as several quaternionic variables.

%


\section{Slice functions of several octonionic variables}

Let $\mathbb{O}_{\mathbb{C}}=\mathbb{O}\otimes_{\mathbb{R}}{\mathbb{C}}$
be  the complexification of the octonions $\mathbb{O}$,  which can also be expressed as
$$\mathbb{O}_{\mathbb{C}}=\mathbb O+i\mathbb O=\{ \omega =x+iy\ \vert \ x,y \in  \mathbb{O}\} \qquad    ( i^2=-1).$$
It is a complex alternative algebra with a unity w.r.t. the product given by the formula
$$(x+iy)(u+iv)=xu-yv+i(xv+yu).$$
For each $x,y\in \mathbb{O},$ we define  $(x+iy)^c=\overline x +i\overline y$ be the complex-linear antiinvolution of $x+iy$ in $\mathbb{O}_{\mathbb{C}}$ and $\overline{x+iy}=x-iy$ be the complex conjugation of $x+iy$ in $\mathbb{O}_{\mathbb{C}}.$

\begin{defn}\label{B}
Let $D$ be an open subset in $\mathbb{C}^n.$ A function $$F: D\rightarrow\mathbb{O}_{\mathbb{C}}$$
is called an $\mathbb{O}$-stem function on $D$, if $F$ is complex intrinsic, i.e.
$$F(\overline{z})=\overline{F(z)}, \qquad \forall \ z\in D.$$
\end{defn}

Notice that by  complex intrinsincity, $F$ can be extended to the  axially symmetric set generated by $D$, i.e.,
$D\cup conj(D)$, where
$$ conj(D):=\{ z\in \mathbb{C}^n \vert\ \overline z \in D\}.$$
However, the extended set may  be non-connected.

\begin{rem}\label{rm-1}
$(1)$  In the case of $n=1$,    Definition \ref{B} was introduced  by Ghiloni and Petrotti  in \cite{Ghiloni2011002} even in any  alternative algebra instead of octonions. Here, we initiate the study to the case of higher dimensions with $n>1$.

$(2)$A function $F=F_1+iF_2$ is a $\mathbb{O}$-stem function if and only if the $\mathbb{O}$-valued components $F_1,F_2$ constitute  an even-odd pair,
i.e.,
$$F_1(\overline{z})=F_1(z), \qquad F_2(\overline{z})=-F_2(z), \qquad \forall\  z\in D.$$

$(3)$ By means of a basis
$\mathcal{B}=\{e_0,\cdots,e_7 \}$ of $\mathbb{O}$
 as a 8-dimensional real vector space.
 $F$ can be identity with a complex intrinsic surface in $\mathbb{C}^8.$
Let $$F(z)=F_1(z)+iF_2(z)=\sum_{k=0}^7 F^k(z)e_k$$ with $F^k(z)\in \mathbb{C}.$
Then$$\tilde F=(F^0,\cdots,F^7):D\rightarrow \mathbb{C}^8 $$ satisfies $$\tilde F(\overline z)= \overline {\tilde F(z)}.$$ Giving $\mathbb{O}$ the unique manifold structure as a real vector space, we get that a stem function $F$
is of class $C^k(k=0,\cdots,\infty)$ or real-analytic if and only if the same property for $\tilde F.$
This notion is clearly independent of the choice of the basis of $\mathbb{O}.$

In several octonionic variables $\mathbb O^n,$ we define  
$$\mathbb{O}_s^n:=\bigcup_{\,J \in\mathbb{S}} \mathbb{C}_J^n,$$
where $\mathbb{S}$ denotes the unit sphere of the imaginary octonions, i.e.,

$$\mathbb{S}=\{a\in \mathbb{O}\ | \ a^2=-1\},$$
and
 $$\mathbb{C}_J^n:=(\mathbb{C}_J)^n$$
with
$$\mathbb{C}_J:=\mathbb{R}+J\mathbb R\subset\mathbb O, \qquad \forall\ J\in \mathbb{S}.$$

\end{rem}	
Given an open subset $D$ of $\mathbb{C}^n,$ let $\Omega_D$ be the subset of $\mathbb{O}_s^n$ generated by $D$:
$$\Omega_D:=\{x=\alpha+\beta J\in \mathbb{C}_J^n \ |\  \forall J\in \mathbb{S},\ \alpha ,\beta \in \mathbb{R}^n\  \text{with}\ \alpha +i \beta \in D \}\subset \mathbb{O}_s^n,$$
and
$$D_J:=\Omega_D \cap \mathbb{C}_J^n,\qquad J\in \mathbb{S}.$$
 
Sets of this type as $\Omega_D$ will be called circular sets in $\mathbb{O}_s^n,$ which is an open subset of $\mathbb{O}_s^n.$

\begin{defn}\label{def-I}A stem function $F:D\rightarrow \mathbb{O}_\mathbb{C}$ induces a (left) slice function $$f=\mathcal{I}(F):\Omega_D \rightarrow \mathbb{O}.$$
 For any  $x=\alpha+J\beta \in D_J,\ \forall J\in\mathbb S$, we set
\begin{eqnarray}\label{eq:slice-fun}f(x):=F_1 (z)+JF_2 (z) \qquad  (z=\alpha + i \beta ).\end{eqnarray}
\end{defn}
The slice function is well defined, since $(F_1,F_2)$ is an even-odd pair w.r.t. $\beta$ and then
$$f(\alpha+(-\beta )(-J))=F_1(\bar {z})+(-J)F_2(\bar {z})=F_1(z)+JF_2(z)=f(\alpha+\beta J).$$

There is an analogous definition for right slice functions when the element $J\in \mathbb{S}$ is placed on the right of $F_2(z).$ From now on,  the term slice functions will always mean left slice function.

We denote the set of $\mathbb O$-stem functions on $D$ as
$$\mathfrak{S}(D, \Omega)
:=\{F: D\rightarrow \mathbb{O}_{\mathbb C}\ |\ F:D\rightarrow \mathbb{O}_\mathbb{C}\ \text{is\ a}\ \mathbb{O}\text{-stem\ function}\}$$
and  denote the set of (left) slice function on $\Omega_D$ by
$$\mathcal{S}(\Omega_D,\mathbb{O}):=\{f:\Omega_D\rightarrow \mathbb{O}\ |\ f=\mathcal{I}(F),\ F:D\rightarrow \mathbb{O}_\mathbb{C}\ \text{is\ a}\ \mathbb{O}\text{-stem\ function}\}.$$
Therefore,
the lift map $\mathcal I$ is a bijection
$$\mathcal I:
\mathfrak{S}(D, \Omega)\longrightarrow \mathcal{S}(D, \Omega).$$

\begin{rem}
$\mathcal{S}(\Omega_D,\mathbb O)$	is a real vector space, since
$$\mathcal{I}(F+G)=\mathcal{I}(F)+\mathcal{I}(G),$$
and
$$\mathcal{I}(aF)=a\mathcal{I}(F)$$
for every complex intrinsic function $F$,\ $G$ on $D$ and $a\in \mathbb{R}$.
\end{rem}

From Definition \ref{def-I}, we obtain the following representation formulas for slice functions. 

\begin{prop}\label{pro-1}
Let $f\in \mathcal{ S} (\Omega_ D, \mathbb O)$, and $J$,$K\in \mathbb{S}$ with $J\not=K$. Then
$$f(\alpha+\beta I)=(I-K)\Big((J-K)^{-1}f(\alpha+\beta J)\Big)-(I-J)\Big((J-K)^{-1}f(\alpha+\beta K)\Big)$$
for each $I\in \mathbb S,\ \alpha, \beta\in\mathbb R^n$ with $\alpha+\beta I\in D_I.$
\end{prop}

\begin{proof}
For any $f\in \mathcal{S} (\Omega_ D, \mathbb O)$, by (\ref{eq:slice-fun}),
$$ f(\alpha +\beta J)-f(\alpha-\beta K)=(J-K)F_2(\alpha+\beta i)$$
Hence $$F_2(\alpha+\beta i)=(J-K)^{-1}(f(\alpha +\beta J)-f(\alpha+\beta K))$$ and \ 	
\begin{equation*}
\begin{split}
F_1(\alpha+\beta i)=&f(\alpha+\beta J)-JF_2(\alpha+\beta i)\\=&f(\alpha+\beta J)-J((J-K)^{-1}(f(\alpha +\beta J)-f(\alpha+\beta K))).
\end{split}
\end{equation*}
Therefore,
\begin {equation*}
\begin {split}
f(\alpha+\beta I )=&F_1(\alpha+\beta i)+IF_2(\alpha+\beta i)
\\=&f(\alpha+\beta J)+(I-J)((J-K)^{-1}(f(\alpha +\beta J)-f(\alpha+\beta K)))\\=&(J-K+I-J)((J-K)^{-1}f(\alpha +\beta J))+(I-J)((J-K)^{-1} f(\alpha+\beta K))\\=&(I-K)((J-K)^{-1}f(\alpha +\beta J))+(I-J)((J-K)^{-1} f(\alpha+\beta K)).
\end{split}
\end{equation*}
\end{proof}
By settting $K=-J$ in Proposition \ref{pro-1}, we obtain the following result.
\begin{cor}
Let $f\in \mathcal{ S} (\Omega_ D,\mathbb O)$ and $J\in \mathbb{S}$. Then 
$$f(\alpha+\beta I )= \frac{1}{2}(f(\alpha+\beta J)+f(\alpha-\beta J))-\frac{I}{2}(J(f(\alpha+\beta J)-f(\alpha-\beta J)))$$
for each  $I\in \mathbb S,\ \alpha, \beta\in\mathbb R^n$ with $\alpha+\beta I\in D_I.$
\end{cor}

When $n=1$, Ghiloni and Petrotti   \cite{Ghiloni2011002} introduced the useful  concepts of the spherical derivative and the spherical value. Now we generalize them to the slice functions of several octonionic variables.

\begin{defn}
Let $f\in \mathcal{ S} (\Omega_ D, \mathbb O).$ The  spherical value of $f$ at $x\in \Omega_D$ is the element of $\mathbb{O}$
$$v_sf(x):=\frac {1}{2}(f(x)+f(\overline x ))$$	
and the spherical derivative of $f$ at $x\in \Omega_ D \backslash \mathbb{R}^n$ is the element of $\mathbb{O}$
$$\partial_s f(x):=\frac {1}{2} \mbox{\it Im}(x) ^{-1} (f(x)-f(\overline x )).$$

In this way, we get two slice functions associated with $f$, given by (\ref{eq:slice-fun}). Namely,  $v_sf$ is induced on $\Omega_ D$ by the stem function $F_1(z)$ and $\partial _s f$ is induced on $\Omega_ D \backslash \mathbb{R}^n$ by $F_2(z).$

Since these stem functions are $\mathbb{O}$ valued,\ $v_sf$ and $\partial_s f$ are constant on every ``sphere"
$$\mathbb{S}_x :=\{ y=\alpha +\beta I \ \vert \ x=\alpha+\beta J,\ \alpha,\ \beta \in \mathbb{R}^n,\ I\in \mathbb{S}\}.$$
Therefore $$\partial_s(\partial_s f)=0, \qquad \partial_s(v_sf)=0$$ \ for every $f.$\ Moreover, $\partial_sf(x)=0$\ if and only if $f$ is constant on $\mathbb{S}_x.$ In this case, $f$ has value $v_sf(x)$ on $\mathbb{S}_x.$

If\ $\Omega_ D \cap \mathbb{R}^n\not =\emptyset,$
under mild reguarity conditions on $F,$ we get that $\partial_sf$ can be contiously extended as a slice function on $\Omega_D.$ For example, it is sufficient to assume that $F_2(z)$ is of class $C^1.$ By definition, the following identity holds for every $x\in \Omega_D:$
$$f(x) =v_sf(x)+\mbox{\it Im}(x) \partial_sf(x).$$
\end{defn}

We will consider slice functions of several octonionic variables induced by  stem functions of class $C^1.$ They consist of  the real vector space
$$\mathcal{S}^1(\Omega_D,\mathbb{O}):=\{f= \mathcal{I}(F) \in\mathcal{S }(\Omega_D,\mathbb{O})|\ F\in C^{1}(D,\mathbb{O}_{\mathbb{C}})\}.$$

Let $f= \mathcal{I}(F) \in\mathcal{S }^1(\Omega_D,\mathbb{O})$ and $z=\alpha+i\beta \in D.$
Then the partial derivatives $\partial F\slash \partial \alpha_t$ and
$i\partial F\slash \partial \beta_t$ are continous $\mathbb{O}-$stem functions on $D$ for any $t=1,2,\cdots,n.$
The same property holds for their linear combinations
$$\frac{\partial F}{\partial z_t}=\frac{1}{2}\left(\frac{\partial  F}{\partial \alpha_t}
-i\frac{\partial F}{\partial \beta_t}\right) \qquad \mbox{and} \qquad \frac{\partial F}{\partial \overline z_t}=\frac{1}{2}\left(\frac{\partial F}{\partial \alpha_t}+i\frac{\partial F}{\partial \beta_t}\right),$$
where $z=(z_1,\cdots,z_n),\ \alpha=(\alpha_1,\cdots,\alpha_n),\ \beta=(\beta_1,\cdots,\beta_n),\ t=1,2,\cdots,n.$

\begin{defn}\label{rm-2}
Let $f= \mathcal{I}(F) \in\mathcal{S }^1(\Omega_D,\mathbb{O}).$ We set
$$\frac{\partial f}{\partial x}:= \mathcal{I}(\frac {\partial F}{\partial z})=\big(\mathcal{I}(\frac {\partial F}{\partial z_1}),\cdots,\mathcal{I}(\frac {\partial F}{\partial z_n})\big),$$
and
$$ \frac{\partial f}{\partial \bar x}:= \mathcal{I}(\frac{\partial F}{\partial \bar z})= \big(\mathcal{I}(\frac {\partial F}{\partial \overline z_1}),\cdots,\mathcal{I}(\frac {\partial F}{\partial \overline z_n})\big)$$
with
$$\frac{\partial f}{\partial x}=\big(\frac{\partial f}{\partial x_1},\cdots,\frac{\partial f}{\partial x_n}\big) \qquad \text{and} 
\qquad \frac{\partial f}{\partial \bar x}=\big(\frac{\partial f}{\partial \overline x_1},\cdots,\frac{\partial f}{\partial \overline x_n}\big).$$
These maps are continous slice maps on $\Omega_D,\ t=1,2,\cdots,n.$
\end{defn}

\section{Slice regular functions of several octonionic variables}

Left multiplication by $i$ defines a complex structure on $\mathbb{O}_\mathbb{C}.$ With respect to this structure, a $C^1$ function $$ F=F_1+iF_2:D\rightarrow \mathbb{O}_\mathbb{C}$$ is holomorphic if and only if its components $F_1,$ $F_2$
satisfy the Cauchy-Riemann equations:
$$\frac{\partial  F_1}{\partial \alpha_t}=\frac{\partial  F_2}{\partial \beta_t},\qquad
\frac{\partial  F_1}{\partial \beta_t}=-\frac{\partial  F_2}{\partial \alpha_t}, \qquad (z=\alpha+i \beta \in D)
$$
i.e., $$ \frac{\partial  F}{\partial \overline z_t}=0,$$
where $z=(z_1,\cdots,z_n),\ \alpha=(\alpha_1,\cdots,\alpha_n),\ \beta=(\beta_1,\cdots,\beta_n),\ t=1,2,\cdots,n.$

This condition is equivalent to require that, for any basis $\mathcal{B},$
the complex surface $\tilde{F}$\ (see Remark \ref{rm-1}) is holomorphic.
Set
\begin{eqnarray*}\frac{\partial F}{\partial z}&:=&\biggl(\frac{\partial F}{\partial  z_1},
\frac{\partial F}{\partial  z_2},\cdots,\frac{\partial F}{\partial z_n}\biggr),
\\
\frac{\partial F}{\partial \bar z}&:=&\biggl(\frac{\partial F}{\partial  \bar z_1},
\frac{\partial F}{\partial  \bar z_2},\cdots,\frac{\partial F}{\partial \bar z_n}\biggr).\end{eqnarray*}

The set of all holomorphic $\mathbb O$-stem functions is denoted by
$$\mathcal{H}(D,\mathbb{O}_{\mathbb{C}}):=\{F\in C^1(D,\mathbb{O}_{\mathbb{C}}):\frac{\partial F}{\partial \overline z}(z)=0, \ \forall z\in D\}.$$

\begin{defn}\label{Def1}
	A (left) slice function $f= \mathcal{I}(F) \in\mathcal{S }^1(\Omega_D,\mathbb{O})$ is (left) slice regular if its associated  stem function $F$ is holomorphic. We will denote the vector space of slice regular functions on $\Omega_D$ by
		$$\mathcal{SR}(\Omega_D,\mathbb{O}):=\{f\in\mathcal{S}^1(\Omega_D,\mathbb{O})\ |\ f= \mathcal{I}(F), F:D\rightarrow \mathbb{O}_\mathbb{C} \ \text {is\ holomorphic}\}.$$
\end{defn}

\begin{rem}
A function $f\in \mathcal{S}^1(\Omega_D,\mathbb{O})$ is slice ragular if and only if the slice map $$\frac{\partial f}{\partial \overline x}=\left(\frac{\partial f}{\partial \overline x_1},\cdots,\frac{\partial f}{\partial \overline x_n}\right)$$\ (cf. Definition \ref{rm-2} in Section 2)  vanishies identically. Moreover, if $f$ is slice regular, then also $$\frac{\partial f}{\partial x}=\left(\mathcal{I}(\frac {\partial F}{\partial z_1}),\cdots,\mathcal{I}(\frac {\partial F}{\partial z_n})\right) $$ is slice regular on $\Omega_D.$
\end{rem}

\begin{prop}
Let	$f= \mathcal{I}(F) \in\mathcal{S }^1(\Omega_D,\mathbb{O}).$ Then $f$ is slice regular on  $\Omega_D$ if and only if the restriction
$$f_J=f\vert _{D_J}:D_J \rightarrow \mathbb{O}$$
is holomorphic for every $J\in \mathbb{S}$ with respect to the complex structures on $D_J$ and $\mathbb{O}$ defined by left multiplication by $J.$
\end{prop}

\begin {proof}
Notice that  $$f_J(\alpha+\beta J)=F_1(\alpha+i\beta)+JF_2(\alpha+i\beta).$$ If $F$ is holomorphic,  then
$$\frac{\partial f_J}{\partial \alpha}+J \frac{\partial f_J}{\partial \beta}=\frac{\partial F_1}{\partial \alpha}+J \frac{\partial F_2}{\partial \alpha}+J(\frac{\partial F_1}{\partial \beta}+J \frac{\partial F_2}{\partial \beta})=0$$
at every point $x=\alpha+J\beta \in D_J.$

Conversely, assume that $f_J$ is holomorphic at every $J\in \mathbb{S}.$ Then
$$0=\frac{\partial f_J}{\partial \alpha}+J \frac{\partial f_J}{\partial \beta}=\frac{\partial F_1}{\partial \alpha}-\frac{\partial F_2}{\partial \beta}+J( \frac{\partial F_2}{\partial \alpha}+\frac{\partial F_1}{\partial \beta} )$$
at every point $z=\alpha+i \beta \in D.$ From the arbitrariness of $J$ it follows that $F_1,F_2$ satisfy the Cauchy-Riemann equations.
\end{proof}

\begin{rem}
The even-odd character of the pair $(F_1, F_2)$ and the proof of preceding proposition show that, in oder to get slice regularity of $f=\mathcal{I}(F)$ with $F\in C^1,$ it is sufficient to assume that two functions $f_J,\ f_K$ \ with $J\not =K$ are holomorphic on domains $D_J$ and $D_K$ respectively (cf. Proposition \ref{pro-1} ). The possibility $K=-J$ is not excluded which  means that the single function $f_J$ must be holomorphic on $D_J.$
\end{rem}

\section{Products of slice functions of several octonionic variables}

In general, the pointwise product of  two slice functions is not a slice function. However, pointwise product in the algebra $\mathbb{O}_\mathbb{C}$ of $\mathbb{O}$-stem functions induces a natural product on slice functions, which is similar to the case of slice function of one variable.

\begin{defn}
Let $f=\mathcal{I}(F),\ g=\mathcal{I}(G)\in \mathcal{S}(\Omega_D,\mathbb{O}).$\ The product of $f$ and $g$ is the slice function
	$$f\cdot g:=\mathcal{I}(FG)\in \mathcal{S}(\Omega_D,\mathbb{O}).$$
\end{defn}

The preceding definition is well-posed, since the pointwise product $$FG=(F_1+iF_2)(G_1+iG_2)=F_1G_1-F_2G_2+i(F_1G_2+F_2G_1)$$ of complex intrinsic functions is still complex intrinsic. It follows directly from the definition that the product is distributive.
The spherical derivative satisfies a Leibniz-type product rule, where evaluation is replaced by spherical value:
$$\partial_s(f\cdot g)=(\partial_sf) (v_sg)+(v_sf)(\partial_sg).$$

\begin{rem}
In general, $$(f\cdot g)(x)\not =f(x)g(x).$$ If $x=\alpha+\beta J$ belongs to $D_J$ and $z=\alpha+i\beta ,$
then $$(f\cdot g )(x)=F_1(z)G_1(z)-F_2(z)G_2(z)+J(F_1(z)G_2(z)+F_2(z)G_1(z),$$
while $$f(x)g(x)=F_1(z)G_1(z)+(JF_2(z))(JG_2(z))+F_1(z)(JG_2(z))+(JF_2(z))G_1(z).$$
If the components $F_1,F_2$ of the first stem function $F$ are real-valued, or if $F$ and $G$ are both $\mathbb{O}$-valued, then $$(f\cdot g)(x) =f(x)g(x), \qquad\forall\ x\in \Omega_D.$$ In this case, we will use also the notation $fg$ in place of $f\cdot g.$
\end{rem}

\begin{defn}
A slice function $f=\mathcal{I}(F)$ is called real if the $\mathbb{O}$-valued components $F_1,F_2$ of its stem function are real valued. Equivalently, $f$ is real if the spherical value $v_sf$ and the spherical derivative $\partial_sf$ are real valued.
\end{defn}

A real slice function $f$ has the characteristic property that for every $J\in \mathbb{S},$ the image $f_J$
is contained in $\mathbb{C}_J.$

\begin{defn}
A slice function $f\in \mathcal{S}(\Omega_D,\mathbb{O})$ is real if and only if $f_J(D_J) \subseteq \mathbb{C}_J $
for every $J\in \mathbb{S}.$
\end{defn}

\begin{proof}
Assume that $f(\mathbb{C}_J^n \cap \Omega_D) \subseteq \mathbb{C}_J$ for every $J\in \mathbb{S}.$ Let $f=\mathcal{I}(F).$
If $x=\alpha+J \beta \in \Omega_D$ and $z=\alpha+i \beta, $ then $$f(x)=F_1(z)+JF_2(z)\in \mathbb{C}_J$$ and
$$f(\bar x)=F_1(\bar z)+JF_2(\bar z)=F_1(z)-JF_2(z) \in \mathbb{C}_J.$$
This implies that $$F_1(z),F_2(z)\in  \bigcap_{J\in \mathbb S} \mathbb{C}_J=\mathbb{R}.$$
\end{proof}

\begin{prop}
If $f,g$ are slice regular on $\Omega_D,$ then the product $f\cdot g$ is slice regular on $\Omega_D.$
\end{prop}

\begin{proof}
Let $f=\mathcal{I}(F),\ g=\mathcal{I}(G),\ H=FG.$ If $F$ and $G$ satisfy the Cauchy-Riemann equations, the same holds for $H.$ This follows from the validity of the Leibniz product rule, that can be checked using a basis representation of $F$ and $G.$
\end{proof}

We consider two polynomials or convergent power series
 $$f(x)=\sum_\mu x^\mu a_\mu, \qquad g(x)=\sum_\nu x^\nu b_\nu, $$
  where $$\mu=(\mu_1,\cdots,\mu_n)\in \mathbb N^n,\qquad  \nu=(\nu_1,\cdots ,\nu_n) \in \mathbb{N}^n, \qquad x^\mu :=x_1^{\mu_1}\cdots x_n^{\mu_n}$$ for $x=(x_1,\cdots,x_n)\in \mathbb{O}_s^n,$  and $$a_\mu:=a_{\mu_1\mu_2\cdots \mu_n}\mathbb O,\qquad b_\nu:=b_{\nu_1,\cdots \nu_n}\in \mathbb{O}.$$
  The star product $f*g$ of $f$ and $g$ is the convergent power series,  defined as
$$f*g:=\sum_\gamma x^\gamma(\sum_{\mu+\nu=\gamma} a_\mu b_\nu).$$

\begin{prop}
Let $f(x)=\sum_\mu x^\mu a_\mu$ and $g(x)=\sum_ \nu x^\nu b_\nu$ be polynomials or convergent power series, where
$a_\mu, b_\nu\in \mathbb{O},\ \mu, \ \nu  \in \mathbb{N}^n.$
Then the product of $f$ and $g$, viewed as slice regular functions, coincides with the star product 
$$ f*g=\mathcal{I}(FG)=\mathcal{I}(F)*\mathcal{I}(G).$$
\end{prop}

\begin{proof}
Let $f=\mathcal{I}(F),\ g=\mathcal{I}(G),\ H=FG$,  and
$$F(z)=\sum_\mu z^\mu a_\mu, \qquad G(z)=\sum_\mu z^\mu b_\mu.$$

Since $\mathbb{R} \otimes_\mathbb{R}\mathbb{C}\simeq \mathbb{C}$ is contained in the commutative and associative center of $\mathbb{O}_\mathbb{C},$ we have
\begin{equation*}
\begin{split}
H(z)=&F(z)G(z)\\=&(\sum z^\mu a_\mu)(\sum z^\nu b_\nu)\\=&\sum_{\mu,\nu}z^\mu z^\nu(a_\mu b_\nu)
\end{split}
\end{equation*}

Denote by $A_\gamma(z)$ and $B_\gamma(z)$ the real components of the complex power
$$z^\gamma=(\alpha+i \beta )^\gamma,\qquad  \alpha, \ \beta \in \mathbb{R}^n,\quad  \gamma \in \mathbb{N}^n.$$

Let $c_\gamma=\sum_{\mu+\nu=\gamma}a_\mu b_\nu$ for each $\gamma.$ Therefore, we have
\begin{equation*}
\begin{split}
H(z)=&\sum_\gamma z^\gamma c_\gamma\\=&\sum_\gamma(A_\gamma(z)+iB_\gamma(z))c_\gamma\\=&\sum_\gamma A_\gamma(z)c_\gamma+ i(\sum_\gamma B_\gamma(z)c_\gamma)\\=:&H_1(z)+iH_2(z)
\end{split}
\end{equation*}
and then, if $x=\alpha+\beta J,\ z=\alpha +i \beta,$
$$\mathcal{I}(H)(x)=H_1(z)+JH_2(z)=\sum_\gamma A_\gamma(z)c_\gamma+ J(\sum_\gamma B_\gamma(z)c_\gamma).$$

On the other hand,
\begin{eqnarray*}
f*g(x)&=&\sum_\gamma (\alpha+\beta J)^\gamma c_\gamma\\
&=&\sum_\gamma (A_\gamma(z)+JB_\gamma(z))c_\gamma
\end{eqnarray*}
since $A_\gamma$ and $B_\gamma$ are all real. From these,  the result follows.
\end{proof}

\section{Zeros of slice functions of several octonionic variables}

The zero sets of slice functions  exhibits many  interesting algebraic and topological properties.
Some relevant theories, concerning the zeros of slice functions of quternionic and octonionic variable, have been studied deeply in
\cite{Gentili2008001, Gentili2008003, Ghiloni2011005, Pogorui2004001}.

The zero set $$V(f)=\{ x\in \mathbb{O}_s^n\ \vert \ f(x)=0\}$$
 of a slice function $f\in \mathcal{S}(\Omega_D,\mathbb{O})$ has a particular structure.
 We will see that, for every fixed $x=\alpha+J\beta \in \mathcal{S}(\Omega_D),$ the ``sphere"
$$\mathbb{S}_x=\{\alpha+I\beta\ \vert  \ I\in \mathbb{S}\}$$
is entirely contained in $V(f)$ or $\mathbb S_x$ contains at most one zero of $f$.

\begin{prop}
Let $f\in \mathcal{S}(\Omega_D,\mathbb{O}).$ For $x\in \Omega_D \backslash \mathbb{R}^n,$ the restriction of $f$ to $\mathbb{S}_x$ is injective or constant.
\end{prop}

\begin{proof}
Given $x,{x}' \in \mathbb{S}_x,$ if $f(x)=f( {x}'),$
then $$(x-{x}')\partial_sf(x)=(\mbox{\it Im}(x)-\mbox{\it Im}({x}'))\partial_sf(x)=0.  $$ If $\partial_sf(x)\not =0,$ this implies $x={x}'.$	

 If $\partial_sf(x) =0,$ then   $\left. f \right|_{\mathbb S_x}$	is a constant due to the representation formula.
\end{proof}

This result leads to a  structure theorem of the zero of $f$ restricted to  the sphere $\mathbb S_x$.

\begin{thm}
(Structure of $V(f) $) Let $f=\mathcal{I}(F) \in \mathcal{S}(\Omega_D,\mathbb{O}).$ Let $x=\alpha +J\beta \in \Omega_D$ and $z=\alpha+i\beta \in D.$
Then one of the following mutually exclusive statements holds:

$(1)\ \mathbb{S}_x \cap V(f)= \emptyset.$

$(2)\ \mathbb{S}_x \subset V(f).$ In this case $x$ is called a real (if $x\in \mathbb{R}^n$) or spherical (if $x\not \in \mathbb{R}^n$) zereo of $f.$

$(3)\ \mathbb{S}_x \cap V(f)$ consists of a single, non-real point.\
In this case $x$ is called a non-real zero of $f$ in $\mathbb S_x$.
\end{thm}

\begin{rem}
We remark that the preceding theorem shows that when restricted to  any sphere $\mathbb S_x$, the zeros of $f$ have the same behavior either $n=1$ or $n>1$.
\end{rem}

\section{Bochner-Martinelli formula  and Hartogs theorem}

The Bochner-Martinelli formula is an important formmula in several complex variables (see Theorem 1.1.4 \cite{Krantz1992001}). We now extend it to  slice functions of several octonionic variables. As an application, we shall see that there appear the Hartogs phenomena when $n>1$ in our setting.

On $\mathbb{C}_J^n$ with any  $J\in\mathbb S$, we consider the Bochner-Matinalli kernel
$$\omega _x(\xi):=\frac {(n-1)!}{(2 \pi J)^n}\sum _{j=1}^n (-1)^{j-1}\frac{\overline{\xi}_j-\overline{x}_j}{|\xi-x|^{2n}}
\overline{d\xi}_1\wedge \cdots \wedge \overline{d\xi}_{j-1}\wedge \overline{d\xi}_{j+1} \wedge \cdots \wedge
\overline{d\xi}_n \wedge d\xi.$$
Here $\xi=(x_1,\cdots,\xi_n)\in\mathbb C_J^n$ and $d\xi:=d\xi_1\wedge\cdots\wedge d\xi_n$.

\begin{thm}\label{thm-bi1}
Assume that  $f\in\mathcal{S}^1(\overline\Omega_D,\mathbb{O})$, $J\in\mathbb{
S}$, and  $D_J$ is a bounded domain with $C^1$ boundary in $\mathbb{C}_J^n$. Then for any $x\in D_J,$
$$f(x)=\int_{\partial D_J}\omega _x(\xi)f(\xi)-(-1)^n\int_{D_J} \omega _(\xi)\wedge \overline{\partial}f(\xi).$$
Moreover,\ for any \ $q\in \Omega_D$ there exists $I\in \mathbb{S}$ such that  $q\in \mathbb{C}_I^n$ and
\begin{equation*}
\begin{split}
f(q)&=\int_{\partial D_J}\frac{1}{2}(\omega _x(\xi)+\omega _{\overline x} (\xi))f(\xi)-\frac{I}{2}(J(\omega _x(\xi)+\omega _{\overline{x}} (\xi))f(\xi)) \\& -(-1)^n\int_{D_J} \frac{1}{2}(\omega _x(\xi)+\omega _{\overline x} (\xi))\wedge \overline{\partial}f(\xi)-\frac{I}{2}(J(\omega _x(\xi)+\omega _{\overline{x}} (\xi))\wedge \overline{\partial}f(\xi)). \end{split}
\end{equation*}

\end{thm}

\begin{proof}
By defintion, for any $f\in \mathcal{S}^1(\overline\Omega_D,\mathbb{O})$   there exists $F\in C^1(D, \mathbb O_{\mathbb C})$ such that $$f=\mathcal{I}(F).$$
Foe any $z=\alpha+\beta i\in D$, we write
 \begin{equation*}
\begin{split}
F(z)&=F_1(z)+iF_2(z)\\&=\sum_{k=0}^7 F_1^k(z)e_k+iF_2^k(z)e_k
\\&=\sum_{k=0}^7 ( F_1^k(z) +iF_2^k(z))e_k \\&=\sum_{k=0}^7 F^k(z)e_k
\end{split}
\end{equation*}
where
  $\{e_0, \cdots e_7 \} $ is a basis of $\mathbb{O}$ and $$F^k=F_1^k+iF_2^k\in C^1({\overline {D}},\mathbb{C}).$$

We abuse of notation by denoting  $\phi_J$   either the isomorphism  $$\phi_J:\mathbb{C}^n \rightarrow \mathbb{C}_J^n $$
or
 the isomorphism $$\phi_j:\mathbb{C} \rightarrow \mathbb{C}_J$$ which sends  $i$ to $J$.

Define $$F_J^k:=\phi_J\circ F^k\circ \phi_J^{-1}\in C^1(\overline D_J,\mathbb{C}_J).$$
Then $$F_J^k(x)=\phi _J\circ F^k(z)=F_1^k(z)+J F_2^k(z),$$
where
$$x=\phi_J(z).$$

From this it follows that
\begin {eqnarray*}
  f(x)&=&\mathcal{I}(F)(x)\\ &=&
 F_1(z)+JF_2(z)\\&=&\sum_{k=0}^7 ( F_1^k(z) +JF_2^k(z))e_k\\&=&\sum_{k=0}^7 F_J^k(x)e_k.
 \end{eqnarray*}
Notice that $F_J^k\in C^1(\overline D_J,\mathbb{C}_J),$ \ and $D_J$ is a bounded domain with $C^1$ boundary in $C_J^n$.
By the  Bochner-Martinelli formula in the function theory of several complex variables, we obtain
$$F_J^k(x)=\int_{\partial D_J}F_J^k(\xi)\omega _x(\xi)-\int_{D_J} \overline{\partial}F_J^k(\xi)\wedge \omega _x(\xi).$$
\\A straight calculation shows that
$$F_J^k(x)=\int_{\partial D_J}\omega _x(\xi)F_J^k(\xi)-(-1)^n \int_{D_J}\omega _x(\xi)\wedge \overline{\partial}F_J^k(\xi). $$
Now we have

\begin {equation*}
\begin{split}
	f(x)&=\sum_{k=0}^7 F_J^k(x)e_k
	\\&=\int_{\partial D_J}(\omega _x(\xi)\sum_{k=0}^7 F_J^k(\xi))e_k-(-1)^n\int_{D_J} (\omega _x(\xi)\wedge \sum_{k=0}^7 \overline{\partial} F_J^k(\xi))e_k
	\\&=\int_{\partial D_J}\omega _x(\xi)(\sum_{k=0}^7 F_J^k(\xi)e_k)-(-1)^n\int_{D_J} \omega _x(\xi)\wedge (\sum_{k=0}^7 \overline{\partial} F_J^k(\xi)e_k)
	\\&=\int_{\partial D_J}\omega _x(\xi)f(\xi)-(-1)^n\int_{D_J} \omega _x(\xi)\wedge \overline{\partial}f(\xi).
\end{split}
\end{equation*}
In the third equation above, we used the alternativity of octonions .
Apply the octonionic representation formula with the function $f$ on the domain $\Omega_D$ (cf. Proposition \ref{pro-1}), we have the other formula.
\end{proof}

As a direct corollary, we get the Cauchy formula for slice regular funtions of sevral octonionic variables.

\begin{cor}\label{thm-bi2}
Let $f\in\mathcal{S}^1(\overline\Omega_D,\mathbb{O})\cap \mathcal{SR}(\Omega _D,\mathbb{O})$,\ $J\in\mathbb{S}$, and
 $D_J$ is a bounded domain with $C^1$ boundary in $\mathbb{C}_J^n$. Then for any $x\in D_J,$
$$f(x)=\int_{\partial D_J}\omega _x(\xi)f(\xi).$$ Moreover,\ for any \ $q\in \Omega_D,$
$$f(q) =\int_{\partial D_J}\frac{1}{2}(\omega _x(\xi)+\omega _{\overline x} (\xi))f(\xi)-\frac{I}{2}(J(\omega _x(\xi)+\omega _{\overline{x}} (\xi))f(\xi)),$$
where $q\in \mathbb{C}_I^n$,\ $\forall I\in \mathbb{S}.$
\end{cor}

By setting $n=1$ in Theorem \ref{thm-bi1}, we obtain the Cauchy integral formula for slice functions of
octonionic variable of class $C^1$, which was obtained by Ghiloni and Perotti\cite{Ghiloni2011002}.


Hartogs's theorem \cite{Hormander1990001} is a fundemental result in the theory of several complex variables. Now we generalize  the Hartogs Theorem to  the case of several octonionic variables.

Let $D$\ be a domain in $\mathbb{C}^n,$ $f\in \mathcal {S}(\Omega_D,\mathbb{O}).$ For any $a\in \mathbb{C}^n,$ we denote
$$D_{j,a}=\{z_j\in \mathbb{C}:(a_1, \cdots, a_{j-1}, z_j, a_{j+1}, \cdots, a_n)\in D\}$$
and
$$\Omega_{D_{j,a}}=\bigcup_{J\in\mathbb S} \phi_J(D_{j,a}).$$

We consider
  the functions on $ D_{j,a} $ $$F_{j,a}(z)=F(a_1, \cdots, a_{j-1}, z_j, a_{j+1}, \cdots, a_n)$$
and its lift  $$f_{j,a}:=\mathcal{I}(F_{j,a})$$
for any $a\in \mathbb{C}^n, \ j=1,\cdots,n,$

\begin{thm}  If for any $a\in \mathbb{C}^n, \ j=1,\cdots,n,$
$$f_{j,a}:=\mathcal{I}(F_{j,a})\in \mathcal{SR}(\Omega_{D_{j,a}},\mathbb{O}), $$
  then $f\in \mathcal{SR}(\Omega_{D}).$
\end{thm}

\begin{proof}
If  $f_{j,a}=\mathcal{I}(F_{j,a})\in \mathcal{SR}(\Omega_{D_j,a},\mathbb{O}),$ which is well-defined (see remark \ref{rm-1}).
From Definition \ref{Def1}, we have $$F_{j,a}\in \mathcal H(D_{j,a},\mathbb{O}_{\mathbb{C}}).$$
This means  that $$F\in \mathcal H(D,\mathbb{O}_{\mathbb{C}})$$ by the Hartogs Theorem for the holomorphic
functions of several complex variables. Hence $f\in \mathcal{SR}(\Omega_{D},\mathbb{O}) $ by definiton.	
\end{proof}

\begin{thm}$(Hartogs)$ Assume that $n\ge2$.
Let $D\subset \mathbb{C}^n$ be a domain, \ $K\subset D$\ be a compact set such that $D\backslash K$ is
connected in $\mathbb{C}^n$. If $f\in \mathcal{SR}(\Omega_ {D\backslash K},\mathbb{O}),$\ then there is a function
$g\in \mathcal{SR}(\Omega_D,\mathbb{O}), $\ such that $g\vert _{\Omega_ {D\backslash K}}=f.$
\end{thm}

\begin{proof}
For any  $f\in \mathcal{SR}(\Omega_ {D\backslash K}, \mathbb O),$\ by definition there exists a function $ F\in \mathcal{H}(D\backslash K, \mathbb{O}_\mathbb{C}),$\ such that $f=\mathcal{I}(F).$
Notice that $$F(z)=\sum _{k=0}^7 F^k(z)e_k,$$ here $\{e_0,\cdots,e_7 \} $ is a basis of $\mathbb{O}$  and $$F^k\in \mathcal{H}(D\backslash K,\mathbb{C}).$$\
By the classical Hartogs theorem for several complex variables, there is a function $G^k\in \mathcal{H}(D,\mathbb{C})$\ such that $G^k\vert _{D\backslash K}=F^k.$
Now set $$G(z):=\sum _{k=0}^7 G^k(z)e_k,$$ then $ G\in \mathcal{H}(D, \mathbb{O}_\mathbb{C})$ so that $g=\mathcal{I}(G),$\
is the desired function.
\end{proof}

\section{Slice functions of several quaternionic variables }

In the previous sections, we consider the slice theory of several octonionic variables. The similar theory holds with the octonions replaced by the quaternions. In this section, the results are stated without proof.

Let $\mathbb{H}_{\mathbb{C}}=\mathbb{H}\otimes_{\mathbb{R}}{\mathbb{C}}$ be the complexification of $\mathbb{H}$, i.e.,
$$\mathbb{H}_{\mathbb{C}}=\{ \omega =x+iy\ \vert \ x,y \in  \mathbb{H}\} \\\    \qquad ( i^2=-1).$$
$\mathbb{H}_{\mathbb{C}}$ is a complex alternative algebra with a unity w.r.t. the product given by the formula $$(x+iy)(u+iv)=xu-yv+i(xv+yu).$$

Since $\mathbb{H}_{\mathbb{C}}=\mathbb{H}+i \mathbb{H},$ in $\mathbb{H}_{\mathbb{C}}$ two commuting operators are defined: the complex-linear antiinvolution $$\omega \mapsto \omega^c=(x+iy)^c=\overline x +i\overline y $$ and the complex conjungation defined by $$\overline \omega=\overline{x+iy}=x-iy.$$

\begin{defn}
Let $D\subset\mathbb{C}^n$ be an open subset. A function $F:D\rightarrow\mathbb{H}_{\mathbb{C}}$ is called a $\mathbb{H}$-stem function, if $F$ is complex intrinsic, i.e. $F(\bar{z})=\overline{F(z)}$, for each $z\in D$.
\end{defn}

\begin{rem}
$(1)F=F_1+iF_2$ is a $\mathbb{H}$-stem function if and only if $F_1,F_2$ form an even-odd pair,
i.e. for each $z\in D,$ we have $F_1(\overline{z})=F_1(z)$, and $F_2(\overline{z})=-F_2(z).$

$(2)$ Consider $\mathbb{H}$ as a 4-dimensional real vector space. By means of a basis $\mathcal{B}=\{e_0,e_1,e_2,e_3\}$ of $\mathbb{H},$ $F$ can be identity with a complex intrinsic surface in $\mathbb{C}^n.$

Let $F(z)=F_1(z)+iF_2(z)=\sum_{k=0}^3 F^k(z)e_k$ with $F^k(z)\in \mathbb{C}.$ Then$$\tilde F=(F^0,F^1,F^2,F^3):D\rightarrow \mathbb{C}^4 $$ satisfies $\tilde F(\bar z)= \overline {\tilde  F(z)}.$ Giving $\mathbb{H}$ the unique manifold structure as a real vector space, we get that a stem function $F$ is of class $C^k(k=0,\cdots,\infty)$ or real-analytic if and only if the same property for $\tilde F.$ This notion is clearly independent of the choice of the basis of $\mathbb{H}.$
\end{rem}

Given an open subset $D$ of $\mathbb{C}^n,$ let $\Omega_D$ be the subset of $\mathbb{H}_s^n$ obtained by the action on $D$ of the square roots of $-1:$
$$\Omega_D:=\{x=\alpha+\beta J\in \mathbb{C}_J^n \ |\  \alpha +i \beta \in D, \alpha ,\beta \in \mathbb{R}^n, \ J\in \mathbb{S}_{\mathbb{H}}\}\subset \mathbb{H}_s^n$$
and
$$D_J=D\cap \mathbb{C}_J^n,$$
where $$\mathbb{S}_{\mathbb{H}}=\{a\in \mathbb{H}\ | \ a^2=-1\}, \qquad \mathbb{H}_s^n=\bigcup_{\,J \in\mathbb{S}_{\mathbb{H}}} \mathbb{C}_J^n.$$
Sets of this type will be called circular sets in $\mathbb{H}_s^n,$ which is an open subset of $\mathbb{H}_s^n.$

\begin{defn}
Any stem function $F:D\rightarrow \mathbb{H}_\mathbb{C}$ induces a (left) slice function $f=\mathcal{I}(F):\Omega_D \rightarrow \mathbb{H}.$ If $x=\alpha+J\beta \in D_J,$ we set $$f(x):=F_1 (z)+JF_2 (z) \qquad  (z=\alpha + i \beta )$$
\end{defn}

We will denote the set of (left) slice function on $\Omega_D$ by
$$\mathcal{S}(\Omega_D,\mathbb{H}):=\{f:\Omega_D\rightarrow \mathbb{H}\ |\ f=\mathcal{I}(F),\ F:D\rightarrow \mathbb{H}_\mathbb{C}\ \text{is\ a}\ \mathbb{H}\text{-stem\ function}\},$$
and
$$\mathcal{S}^1(\Omega_D,\mathbb{H}):=\{f=\mathcal{I}(F)\in \mathcal{S}(\Omega_D,\mathbb{H})\ | \ F\in C^1(D,\mathbb{H}_\mathbb{C}),$$

$$\mathcal{SR}(\Omega_D,\mathbb{H}):=\{f=\mathcal{I}(F)\in \mathcal{S}(\Omega_D,\mathbb{H})\ | \ F:D\rightarrow \mathbb{H}_\mathbb{C}\ \text{is\ holomorphic} \}.$$

As in the section above, we let
$$\omega _x(\xi):=\frac {(n-1)!}{(2 \pi J)^n}\sum _{j=1}^n (-1)^{j-1}\frac{\overline{\xi}_j-\overline{x}_j}{|\xi-x|^{2n}}\overline{d\xi}_1\wedge \cdots \wedge \overline{d\xi}_{j-1}\wedge \overline{d\xi}_{j+1} \wedge \cdots \wedge \overline{d\xi}_n \wedge d\xi$$
on $\mathbb{C}_J^n.$

\begin{thm}\label{thm-bi3}
Let $f\in\mathcal{S}^1(\overline\Omega_D,\mathbb{H})$,\ $J\in\mathbb{S}_{\mathbb{H}}$.\
Assume that $D_J$ be a bounded domain with $C^1$ boundary in $\mathbb{C}_J^n$, \ then for any $x\in D_J,$
$$f(x)=\int_{\partial D_J}\omega _x(\xi)f(\xi)-(-1)^n\int_{D_J} \omega _x(\xi)\wedge \overline{\partial}f(\xi).$$
Moreover,\ for any \ $q\in \Omega_D,$
\begin{equation*}
\begin{split}
f(q)=\int_{\partial D_J} \omega _q(\xi)f(\xi))-(-1)^n\int_{D_J} \omega _q(\xi)\wedge \overline{\partial}f(\xi))
\end{split}
\end{equation*}
where $$\omega _q(\xi)= \frac{1}{2}(\omega _x(\xi)+\omega _{\overline x} (\xi)) -\frac{I}{2}(J(\omega _x(\xi)+\omega _{\overline{x}} (\xi)), \qquad
q\in \mathbb{C}_I^n,  \quad I\in \mathbb{S}_{\mathbb{H}}.$$
\end{thm}

\begin{thm}\label{thm-bi4}
Let $f\in\mathcal{S}^1(\overline\Omega_D,\mathbb{H})\cap \mathcal{SR}(\Omega _D,\mathbb{H})$,\ $J\in\mathbb{
S}$.\ Assume that $D_J$ be a bounded domain with $C^1$ boundary in $\mathbb{C}_J^n$, \ then for any $x\in D_J,$
$$f(x)=\int_{\partial D_J}\omega _x(\xi)f(\xi) .$$
Moreover,\ for any \ $q\in \Omega_D,$
$$f(q) =\int_{\partial D_J}\omega _q(\xi)f(\xi),$$
where $$\omega _q(\xi)= \frac{1}{2}(\omega _x(\xi)+\omega _{\overline x} (\xi)) -\frac{I}{2}(J(\omega _x(\xi)+
\omega _{\overline{x}} (\xi)),  \qquad\ \forall\ q\in \mathbb{C}_I^n, \quad  I\in \mathbb{S}_{\mathbb{H}}.$$
\end{thm}


	

Let $D$\ be a domain in $\mathbb{C}^n$. For any $a=(a_1,a_2,\cdots,a_n)\in D$ and $j=1,\ldots, n$, we consider the slices of $D$ defined by
$$D_{j,a}=\{u\in \mathbb{C}:(a_1, \cdots, a_{j-1}, u, a_{j+1}, \cdots, a_n)\in D\}\subset \mathbb C$$
and the related sets
$$\Omega_{D_{j,a}}=\{\alpha_j+\beta_j J\in\mathbb H:  \alpha+i\beta\in D, \alpha,\beta\in\mathbb R^n, J\in\mathbb S_{\mathbb H}\}\subset \mathbb H$$
Here we denote
$\alpha=(\alpha_1,\ldots,\alpha_n)$, $\beta=(\beta_1, \ldots, \beta_n).$

Associated with the functions of $f\in \mathcal{S}(\Omega_{D},\mathbb{H})$ and $F\in \mathfrak{S}(\Omega_{D},\mathbb{H}_{\mathbb C})$, we consider their slices
\begin{eqnarray*} F_{j,a}(u)&=&F(a_1, \cdots, a_{j-1}, u, a_{j+1}, \cdots, a_n), \qquad \forall\ u\in  D_{j,a};
\\
f_{j,a}(x_j)&=&f(a_1, \cdots, a_{j-1}, x_j, a_{j+1}, \cdots, a_n), \qquad \forall\ x_j\in \Omega_{D_{j,a}}\subset \mathbb H.
\end{eqnarray*}

It is easy to see that if $f=\mathcal I(F)$ then
$$f_{j,a}:=\mathcal{I}(F_{j,a}).$$
Moreover, if  $f_{j,a}:=\mathcal{I}(F_{j,a})\in \mathcal{SR}(\Omega_{D_{j,a}},\mathbb{H})$ \ for any $a\in D\subset \mathbb{C}^n, \ j=1,\cdots,n,$ then $$f\in \mathcal{SR}(\Omega_{D},\mathbb{H}).$$

\begin{thm}
Let $D$\ be a domain in $\mathbb{C}^n$ and $f\in \mathcal {S}(\Omega_D,\mathbb{H}).$
If for any $a\in D\subset \mathbb{C}^n$ and $j=1,\cdots,n,$
$$f_{j,a}\in \mathcal{SR}(\Omega_{D_{j,a}},\mathbb{H}),$$ \ then $$f\in \mathcal{SR}(\Omega_{D},\mathbb{H}).$$
\end{thm}

Now we can state  the  Hartogs theorem in the version of  several quaternionic variables.

\begin{thm}
Assume that $n\ge2$.
Let $D\subset \mathbb{C}^n$ be a domain, \ $K\subset D$\ be a compact set such that $D\backslash K$ is
connected in $\mathbb{C}^n$. If $f\in \mathcal{SR}(\Omega_ {D\backslash K},\mathbb{H}),$\ then there is a function
$g\in \mathcal{SR}(\Omega_D,\mathbb{H}), $\ such that $g\vert _{\Omega_ {D\backslash K}}=f.$
\end{thm}

\section{Conclusions}

We initiate the study of   the theory of slice regular  functions of sevevral octonionic variables as well as  several quaternionic variables.
The related Bochner-Martinelli formula and the Hartogs theorem are established. It deserves to consider  further extensions of the classical theory of   several complex variables to these new settings.

\newpage

\nocite{*}
\bibliographystyle{plain}
\bibliography{mybibfile.bib}
\end{document}